\documentclass[12pt]{amsart}
\usepackage[dvips]{graphicx}
\usepackage{amsmath,graphics}
\usepackage{amsfonts,amssymb}
\usepackage{xypic,hyperref}
\usepackage{comment}
\usepackage{color}
\specialcomment{proofc}{}{}
\includecomment{proofc}
\theoremstyle{plain}
\newtheorem*{theorem*}{Theorem}
\newtheorem*{lemma*} {Lemma}
\newtheorem*{corollary*} {Corollary}
\newtheorem*{proposition*}{Proposition}
\newtheorem*{conjecture*}{Conjecture}
\newtheorem{theorem}{Theorem}[section]
\newtheorem{lemma}[theorem]{Lemma}
\newtheorem*{theorem1*}{Theorem 1}
\newtheorem*{theorem2*}{Theorem 2}
\newtheorem*{theorem3*}{Theorem 3}

\newtheorem{proposition}[theorem]{Proposition}

\newtheorem{question}[theorem]{Question}

\newtheorem{ther}{Theorem}

\newtheorem{prop}[ther]{Proposition}
\theoremstyle{remark}
\newtheorem*{remark}{Remark}
\newtheorem*{remarks}{Remarks}

\newtheorem*{example*}{Example}
\newtheorem*{claim}{Claim}

\theoremstyle{definition}

\def\op{\operatorname}

\def\G{\Gamma}

   \def\Z{\Bbb{Z}} \def\R{\Bbb{R}} \def\C{\Bbb{C}} 
    
 \def\a{\alpha}  \def\tor{\mbox{Tor}} \def\bp{\begin{pmatrix}}
\def\sm{\setminus} \def\ep{\end{pmatrix}} \def\bn{\begin{enumerate}} \def\Hom{\mbox{Hom}}
   \def\en{\end{enumerate}}
\def\ba{\begin{array}} \def\ea{\end{array}}  
 \def\S{\Sigma}  \def\a{\alpha} \def\b{\beta}

\def\ker{\mbox{Ker}}\def\be{\begin{equation}} \def\ee{\end{equation}}

    \def\rk{\op{rk}}

\def\alb{\mbox{alb}_{X} \colon X \to \mbox{Alb}(X)}

\def\op{\operatorname}
\def\orb{\op{orb}}
\def\alb{\op{Alb}}

\def\jac{\op{Jac}}

\begin{document}
\title{Virtual algebraic fibrations of K\"ahler groups}

\author{Stefan Friedl}
\address{Fakult\"at f\"ur Mathematik, Universit\"at Regensburg, Germany}
\email{sfriedl@gmail.com}

\author{Stefano Vidussi}
\address{Department of Mathematics, University of California,
Riverside, CA 92521, USA} \email{svidussi@ucr.edu}
\date{\today}

\begin{abstract} This paper stems from the observation (arising from work of T. Delzant) that ``most" K\"ahler groups $G$ virtually  algebraically fiber, i.e. admit a finite index subgroup that maps onto $\Z$ with finitely generated kernel. For the remaining ones,  the Albanese dimension of all finite index subgroups is at most one, i.e. they  have  \textit{virtual} Albanese dimension $va(G) \leq 1$. We show that the existence of algebraic fibrations has implications in the study of coherence and  higher BNSR invariants of the fundamental group of aspherical K\"ahler surfaces. The class of K\"ahler groups with $va(G)  = 1$ includes virtual surface groups. Further examples exist; nonetheless, they exhibit a strong relation with surface groups. In fact, we show that the Green--Lazarsfeld sets of groups with $va(G) = 1$ (virtually) coincide with those of surface groups, and furthermore that the only virtually RFRS groups with $va(G) = 1 $ are virtually surface groups.  
\end{abstract}

\maketitle

\section{Introduction} \label{sec:intro}
This paper is devoted to the study of some virtual properties of K\"ahler groups, i.e.\ fundamental groups of compact K\"ahler manifolds. (Recall that if $\mathcal{P}$ is a property of groups, we say that a group $G$ is \textit{virtually} $\mathcal{P}$ if a finite index subgroup $H \leq_{f} G$ is $\mathcal{P}$.)

The guiding principle is to understand if some of the virtual properties of fundamental groups of irreducible $3$--manifolds with empty or toroidal boundary,  recently emerged from the work of Agol, Wise \cite{Ag13,Wi09,Wi12}  and their collaborators, have a counterpart for K\"ahler groups. Admittedly, there is no \textit{a priori} geometric reason to expect any analogy. However this approach seems to be not completely fruitless:  in \cite{FV16} we investigated some consequences of the fact that both classes of groups satisfy a sort of ``relative largeness" property, namely that any epimomorphism $\phi\colon  G \to \Z$ with {\em infinitely generated} kernel virtually factorizes through an epimorphism to a free nonabelian group. (This is a nontrivial result for both classes.) 

In this paper we investigate, in a sense, the opposite phenomenon, namely the existence of epimorphisms $\phi\colon  G \to \Z$ with {\em finitely generated} kernel. Or, stated otherwise,  whether $G$ is an extension of $\Z$ by a finitely generated subgroup. This condition has been conveniently referred to  in \cite{JNW17} by saying that $G$ \textit{algebraically fibers} and here we will adhere to that terminology. Again, thanks to the  recent results of Agol, Wise and collaborators (see \cite{AFW15} for accurate statements and references) the emerging picture is that ``most" freely indecomposable $3$--manifold groups (e.g.\ hyperbolic groups) virtually algebraically fiber. This result has triggered recent interest in the  study of algebraic fibration for various classes of groups, and relevant results have appeared, including during the preparations of this manuscript, see \cite{FGK17,JNW17,Ki18}.

We have tasked ourselves with the purpose of understanding virtual algebraic fibrations in the realm of K\"ahler groups.

The starting point amounts to reinterpreting geometrically Delzant's results on the Bieri--Neumann--Strebel  invariants of K\"ahler groups (\cite{De10}), and is possibly known at least implicitly to those familiar with that result. In particular, it entails a geometric characterization of K\"ahler manifolds whose fundamental group does not virtually algebraically fiber. To state this result, recall that the Albanese dimension $a(X) \geq 0$ of a K\"ahler manifold is defined as the complex dimension of the image of $X$ under the Albanese map $\alb$. We define the \textit{virtual} Albanese dimension $va(X)$ to be the supremum of the Albanese dimension of all finite covers of $X$. (This definition replicates that of virtual first Betti number $vb_1$, that will be as well of use in what follows.)   The property of having Albanese dimension equal to zero or equal to one is  determined by the fundamental group $G = \pi_1(X)$ alone (see e.g.\ Proposition \ref{prop:samealb}); because of that, it makes sense to talk of (virtual) Albanese dimension of a K\"ahler group as an element of $\{0,1,>1\}$.
 With this in mind we have the following:

\begin{ther} \label{thm:main} Let $G$ be a K\"ahler group. Then either $G$ virtually algebraically fibers, or $va(G) \leq 1$.
\end{ther} 

This statement is, in essence, an alternative: the intersection are groups $G$ that have a finite index subgroup  $H \leq _{f} G$ with $b_1(H) = vb_1(G) = 2$
 such that the commutator subgroup $[H,H]$ is finitely generated. (Such a $H$ appears as fundamental group of a genus $1$ Albanese pencil without  multiple fibers.) We could phrase this theorem as an alternative, but the form above fits well with what follows.

Theorem \ref{thm:main} kindles interest in identifying the class of K\"ahler groups with $va(G) \leq 1$, and in what follow we summarize what we know about this class.

To start, a K\"ahler group has $va(G) = 0$ if and only if all its finite index subgroup have finite abelianization, or equivalently $vb_1(G) = 0$. All finite groups fall in this class, but there exist infinite examples as well. An example that is easy to describe is $Sp(2n,\Z)$ for $n \geq 2$ (see \cite{To90}). Also, all cocompact lattices in hermitean symmetric spaces of higher rank belong to this class.

The class of K\"ahler groups with  $va(G) = 1$ contains some obvious examples, namely virtual surface groups (i.e.\ fundamental groups of compact Riemann surfaces of positive genus). In general, the properties of the Albanese map give a tight relation between K\"ahler groups with $va(G) = 1$ and surface groups: When a K\"ahler manifold $X$  has Albanese dimension one, it is well--known that the Albanese map has smooth image and connected fibers (see e.g. \cite[Proposition I.13.9]{BHPV04}), hence determines a genus $g = q(X)$ pencil  $f \colon X \to \Sigma$, referred to as \textit{Albanese pencil}. (Here $q(X) := \frac{1}{2}b_1(X)$ denotes the irregularity of $X$.) When $G = \pi_1(X)$ satisfies $va(G)  = 1$, the Albanese pencil of $X$ lifts to an Albanese pencil ${\tilde f}\colon {\widetilde X} \to {\widetilde \Sigma}$ for all finite covers ${\widetilde X} \to {X}$.

Virtual surface groups do not exhaust the class of groups with $va(G) = 1$: it is not hard to give further (albeit unsophisticated) examples by taking the product of a surface group with $Sp(2n,\Z)$, $n \geq 2$. That said, we are not aware of any subtler construction that does not hinge on the existence of K\"ahler groups with $vb_{1} = 0$. It would interesting to decide if such constructions exist. 

We analyze in Section \ref{sec:va1} the implications of the existence of algebraic fibrations in the context of aspherical K\"ahler surfaces. This allows us to refine some results about coherence of their fundamental group, which appear in \cite{Ka98,Ka13,Py16}. (Recall that a group is coherent if all its finitely generated subgroups are finitely presented.) Combining these results with ours yields the following: 

\begin{ther} \label{thm:co} Let $G$ be a group with $b_1(G) > 0$  which is the fundamental group of  an aspherical K\"ahler surface $X$; then $G$ is not coherent, except for the case where it is virtually the product of $\Z^2$ by a surface group, and perhaps for the case where $X$ is finitely covered by a Kodaira fibration of virtual Albanese dimension one. \end{ther}

(A {\em Kodaira fibration} is a smooth non-isotrivial pencil of curves: in particular, fibers and base have genera at least $3$ and $2$ respectively.) We are not aware of the existence of Kodaira fibrations of virtual Albanese dimension one (Question \ref{qu:bq}).

The techniques used in the proof of the theorem above actually entails the existence of K\"ahler groups whose second Bieri-Neumann-Strebel-Renz (henceforth BNSR) invariant is {\em strictly} contained in the first (see Proposition \ref{proposition:bnsr}) and are not a direct product. We state here a striking instance of this:

\begin{prop} \label{prop:ak} Let $X$ be one of the Kodaira fibrations described by Atiyah and Kodaira; then the BNSR invariants of $G = \pi_1(X)$ satisfy \[  \S^2(G) \subsetneq \S^1(G) \subsetneq S(G). \] \end{prop}

This (and Proposition \ref{proposition:bnsr}) may well be the first results on higher BNSR invariants of nontrivial extensions of surface groups by surface groups: see also \cite[Question 11(4)]{Hil15}, which raises the issue of coherence for this class of groups. Remarkably, the proof of Proposition \ref{prop:ak} uses in crucial manner the fact that $G$ is K\"ahler: we are not aware of any other means to show that $\S^1(G)$ is nonempty for other surface bundles. Note that as we discuss in Section  \ref{sec:va1} the fundamental group of $X$ (and other double fibered Kodaira fibrations) injects in the mapping class groups of the once--punctured fibers.

When $va(G)  =1$, the relation between $X$ and $\Sigma$ induced by the Albanese pencil $f \colon X\to \S$ is stronger than the isomorphism of the first cohomology groups. Indeed, it entails a relation between the Green--Lazarsfeld sets of their (orbifold) fundamental groups (whose definition we recall in Section \ref{sec:va1}): Given an Albanese pencil, we refer to the induced map on (orbifold) fundamental groups $f \colon G \to \G$ (where $G := \pi_1(X)$ and  $\G := \pi^{\orb}_1(\S)$) as Albanese map as well, and we have the following: 

\begin{ther} \label{thm:gli} Let $G$ be a group with $va(G) = 1$.
Perhaps  after going to a finite index normal subgroup,  the Albanese map $f\colon G \to \G$ induces an isomorphism of the Green--Lazarsfeld sets
\[ {\widehat f}\colon {W}_{i}({\G} ) \stackrel{\cong}{\longrightarrow}  {W}_{i}(G).\]  \end{ther} 

Proceeding in another direction, namely restricting the class of K\"ahler groups by imposing residual properties that mirror those of $3$--manifold groups, we obtain a refinement to Theorem \ref{thm:main} that can be thought of as an analogue (with much less work on our side) to Agol's virtual fiberability result  (\cite{Ag08}) for $3$--manifold groups that are virtually \textit{residually finite rationally solvable} (henceforth RFRS, see Section \ref{sec:va1} for the definition).

\begin{ther}  \label{thm:vrfrs}  Let $G$ be a K\"ahler group that is virtually RFRS. Then either $G$  virtually algebraically fibers, or it is virtually a surface group.  \end{ther} 

To date, the infinite K\"ahler groups known to be RFRS are subgroups of the direct product of surface groups and abelian groups. This is a remarkable but not yet completely understood class of K\"ahler groups: in \cite[7.5 sbc--Corollary]{DG05} some conditions for a K\"ahler group to be virtually of this type are presented. Further results on this class appear in \cite{DPS09} and \cite{LI17}. 

We should add that the result of Theorem \ref{thm:vrfrs} can now be thought as a consequence of the recent results of \cite{Ki18}, using some standard facts on K\"ahler groups.

\subsection*{Structure of the paper} Section \ref{sec:ref} discusses some preliminary results on the Albanese dimension of K\"ahler manifolds and groups, as well as the proof of Theorem \ref{thm:main}. Section \ref{sec:va1} is devoted to the study of groups of virtual Albanese dimension one, and contains the proofs of Theorems \ref{thm:co}, \ref{thm:gli}, and \ref{thm:vrfrs}.

 In order to keep the presentation reasonably self--contained, we included some fairly classical results, for which we could not find a formulation in the literature suitable for our purposes. 


\subsection*{Acknowledgment.} The authors want to thank Thomas Delzant, Dieter Kotschick, Mahan Mj, and Matt Stover for their comments and clarifications to previous versions of this paper, as well as the anonymous referee. Also, they gratefully acknowledge the support provided by the SFB 1085 `Higher
Invariants' at the University of Regensburg, funded by the Deutsche
  Forschungsgemeinschaft (DFG), and by the Simons Foundation Collaboration Grant For Mathematician 524230.

\section{Albanese dimension and algebraic fibrations} \label{sec:ref}
We start with some generalities on pencils on K\"ahler manifolds and the properties of their fundamental group. The reader can find in the monograph \cite{ABC$^+$96} a detailed discussion of K\"ahler groups, and the r\^ole they play in determining pencils on the underlying K\"ahler manifolds.
Given a genus $g$ pencil on $X$ (i.e.\ a surjective holomorphic map
 with connected fibers $f \colon X \to \S$  to a surface with $g = g(\S)$) we can consider the homotopy--induced epimorphism $f \colon \pi_1(X)  \to \pi_1(\S)$. In presence of multiple fibers we have a factorization 
$f \colon \pi_{1}(X) \to \pi_1^{\orb}(\S) \to \pi_1(\S)$ through a further epimorphism onto $\pi_1^{\orb}(\S)$, the orbifold fundamental group of $\S$ associated to the pencil $f$, with orbifold points and multiplicities corresponding to the multiple fibers of the pencil. throughout the paper we write  $G := \pi_1(X)$ and $\G := \pi_1^{\orb}(\S)$.  The factor epimorphism, that by slight abuse of notation we denote as well as $f \colon G \to \G$,  has finitely generated kernel so we have the short exact sequence of finitely generated groups \[  1 \rightarrow K \rightarrow  G \stackrel{f}\rightarrow \G \rightarrow 1 \]
(see e.g.\ \cite{Cat03} for details of the above). 

We have the following two results about K\"ahler manifolds of (virtual) Albanese dimension one. These are certainly well--known to the experts (at least implicitly), and we provide proofs for completeness.

\begin{proposition} \label{prop:samealb} Let $X$ be a K\"ahler manifold and let $G = \pi_1(X)$ be its fundamental group. If $X$ has Albanese dimension $a(X) \leq 1$, any K\"ahler manifold with isomorphic fundamental group  has the same Albanese dimension as $X$. \end{proposition}
\begin{proof} The case where $a(X) = 0$ corresponds to manifolds with vanishing irregularity $q(X) = \frac12 b_{1}(X)$ so it is determined by the fundamental group alone. Next, consider a K\"ahler manifold $X$ with positive irregularity.  For any such $X$, the \textit{genus} $g(X)$ is defined as the maximal rank of submodules of $H^1(X)$ isotropic with respect to the cup product. In \cite[Chapter 2]{ABC$^+$96} it is shown that $g(X)$ is in fact an invariant of the fundamental group alone.  As the cup product is nondegenerate, there is a bound \[ g(X) \leq q(X) = \frac12 b_1(X) = \frac12 b_1(G). \] By Catanese's version of the Castelnuovo--de Franchis theorem (\cite{Cat91}), the case where $X$ has Albanese dimension one occurs  exactly for $g(X) = q(X)$. By the above, this equality is determined by the fundamental group alone. \end{proof}

Based on the observations above, we will refer to the Albanese dimension of a K\"ahler  group as an element of the set $\{0,1,>1\}$.  
  
Whenever $G$ has Albanese dimension one, the kernel of the map $f_{*} \colon H_1(G) \to H_1(\G)$ induced by the (homotopy) Albanese map, identified by the Hochschild--Serre spectral sequence with a quotient of the coinvariant homology $H_1(K)_{\G}$ by a torsion group, is torsion (or equivalently $f^{*} \colon H^1(\G) \to H^1(G)$ is an isomorphism).
 Note that, by universality
  of the Albanese map, whenever a K\"ahler manifold $X$ has a pencil $f \colon X \to \S$ such that the kernel of  $f_{*} \colon H_1(G) \to H_1(\G)$ is torsion, the pencil is Albanese.

Let $\pi \colon {\widetilde X} \to X$ be a finite cover of $X$. Denote by $H \leq_{f} G$ the subgroup associated to this cover; the regular cover of $X$ determined by the normal core $N_{G}(H) = \bigcap_{g \in G} g^{-1}Hg \leq_{f} H$ is a finite cover of ${\widetilde X}$ as well. By universality of the Albanese map, the Albanese dimension is nondecreasing when we pass to finite covers.  Therefore (as happens with virtual Betti
 numbers) we can define the virtual Albanese dimension of $X$ in terms of  finite \textit{regular} covers. Given an epimorphism onto a finite group $\a\colon \pi_1(X) \to S$ we have the commutative diagram (with self--defining notation) 
\begin{equation} \label{eq:diag} \xymatrix@R0.5cm{
& 1 \ar[d] &
1 \ar[d] & 1 \ar[d] &\\
1 \ar[r] & \Delta \ar[d] \ar[r] &
H \ar[d] \ar[r]& \Lambda \ar[d] \ar[r]&1\\
1\ar[r]&
K \ar[r]\ar[d]^{\a}&
G \ar[d]^{\a}\ar[r]& \G \ar[d] \ar[r]&1\\
 1\ar[r]& \a(K) \ar[d]
 \ar[r]& S \ar[r] \ar[d]  &
S/ \a(K) \ar[r] \ar[d] &1 \\
& 1 & 1 & 1} 
\end{equation}
Denote by ${\widetilde X}$ and ${\widetilde \Sigma}$ the induced covers of $X$ and $\Sigma$ respectively (so that $\pi_{1}({\widetilde X}) = H$ and $\pi_{1}({\widetilde \Sigma}) = \Lambda$). There exists a pencil ${\tilde f}\colon {\widetilde X} \to {\widetilde \Sigma}$, which is a lift of $f\colon X \to \Sigma$; in homotopy, this corresponds to the epimorphism ${\widetilde f} \colon H \to  \Lambda := \pi_{1}^{\orb}({\widetilde \S})$ appearing in (\ref{eq:diag}) above.

In the next proposition, we illustrate the fact that when $X$ has Albanese dimension one, its Albanese pencil
 $f \colon X \to \S$ is the only irrational pencil
 of $X$, up to holomorphic automorphisms of $\S$. 

\begin{proposition} \label{prop:unique} Let $X$ be a K\"ahler manifold with $a(X) = 1$. Then the Albanese pencil $f \colon X\to \S$  is the unique irrational pencil on $X$ up to holomorphic automorphism of the base. Moreover, if $X$ satisfies $va(X) = 1$, any rational pencil has orbifold base with finite orbifold fundamental group.
 \end{proposition} 

\begin{proof} By assumption the Albanese pencil $f \colon X \to \S$ factorizes the Albanese map $\alb$. Let $g \colon X \to \S'$ be an irrational pencil, and compose it with the Jacobian map $j \colon \S' \to \jac(\S')$. By universality of the Albanese map we have the  commutative diagram
\[ \xymatrix{  X \ar[r]^f \ar[dr]_g  & \S  \ar@{-->}[d]^h  \ar[r] & \alb(X) \ar[d] \\ 
& \S' \ar[r]^{j} & \jac(\S'). } \]
The map $h \colon \S \to \S'$ is well defined by injectivity of the Jacobian map, and is a holomorphic surjection by universality of the Albanese map. Holomorphic surjections of Riemann surfaces are ramified covers; however, unless the cover is one--sheeted, i.e.\ $h$ is a holomorphic isomorphism, the fibers of $g \colon X \to \S'$ will fail to be connected.  

This argument above does not prevent $X$ from having rational pencils. However, if $X$ has also virtual Albanese dimension one, this imposes constraints on the multiple fibers of those pencils. Recall that orbifolds with infinite $\pi_1^{\orb}(\S)$ are those that are
 flat or hyperbolic, hence admit a finite index normal subgroup that is a surface group with positive $b_1$ (see \cite{Sc83} for this result and a characterization of these orbifolds in terms of the singular points). Therefore, if $X$ were to admit a rational pencil $g \colon X \to \Sigma'$ 
with infinite $\pi_1^{\orb}(\S')$, there would exist an irrational pencil without multiple fibers ${\widetilde g} \colon {\widetilde X} \to {\widetilde \S'}$ covering $g \colon X \to \S'$. We claim that the pencil  ${\widetilde g} \colon {\widetilde X} \to {\widetilde \S'}$ cannot be Albanese: building from the commutative diagram in (\ref{eq:diag}) we have the commutative diagram \[ \xymatrix@=9pt{ 
 1\ar[rr] & & K \ar[rr]\ar[dd]_{\cong}\ar[dr] & & \pi_1({\widetilde X}) \ar'[d][dd]\ar[dr]  \ar[rr]^{\widetilde g} & & \pi_1({\widetilde \S'})\ar[rr]\ar[dr]\ar'[d][dd] & & 1
\\ & & & H_1(K) \ar[dd] \ar[rr] & & H_1({\widetilde X}) \ar[rr]^{\hspace{-4mm}\widetilde g}\ar[dd] &  & H_{1}({\widetilde \S'})
\\ 1 \ar[rr] &  & K \ar'[r][rr]\ar[dr] & & \pi_1(X)  \ar'[r][rr]\ar[dr] & & \pi_1^{\orb}(\S') \ar[rr] & & 1
\\ & & & H_1(K) \ar[rr] & & H_1(X)  & &  } \] The subgroup $\mbox{im}(H_1(K) \to H_1(X)) \leq H_1(X)$ has positive rank (it is a finite index subgroup), hence by commutativity the image $\mbox{im}(H_1(K) \to H_1({\widetilde X})) \leq H_1({\widetilde X})$ has positive rank. It follows that the kernel of the epimorphism ${\widetilde g} \colon H_1({\widetilde X}) \to H_1({\widetilde \S'})$ is not torsion, so ${\widetilde g} \colon {\widetilde X} \to {\widetilde \S'}$ is not Albanese. As ${\widetilde X}$ has Albanese dimension one, this is inconsistent with the first part of the statement. \end{proof}

We are now in a position to prove Theorem \ref{thm:main}. In order to do so, it is both practical and insightful to use the Bieri--Neumann--Strebel invariant of a finitely presented group $G$ (henceforth BNS), for which we refer to \cite{BNS87}. This invariant is an open subset $\Sigma^1(G)$ of the character sphere
 $S(G) := (H^1(G;\R)\sm \{0\})/ \R_{>0}$  of $H^1(G;\R)$. For rational rays, the invariant can be described as follows: a rational ray in  $S(G)$ is determined by a primitive class $\phi \in H^1(G)$. Given such $\phi$, we can write $G$ as Higman--Neumann--Neumann (henceforth HNN) extension $G = \langle A,t|t^{-1}Bt = C\rangle$ for some finitely generated subgroups $B,C \leq A \leq \ker ~ \phi$ with $\phi(t)=1$.
The extension is called ascending (descending) if $A = B$  (resp. $A = C$). By \cite[Proposition 4.4]{BNS87} the extension is  ascending (resp. descending) if and only if the rational ray determined by $\phi$ (resp. $-\phi$) is contained in $\Sigma^1(G)$.

We have the following lemma, which applied to the collection of finite index subgroups of $G$,  implies Theorem \ref{thm:main}:

\begin{lemma} \label{lemma:equivalent} Let $G$ be a  K\"ahler group.  The following are equivalent:
\bn 
\item $G$ algebraically fibers;
\item The BNS invariant  $\Sigma^1(G) \subseteq S(G)$ is nonempty;
\item For any compact K\"ahler manifold $X$ such that $\pi_{1}(X)= G$ either the Albanese map is a genus $1$ pencil without multiple fibers, or $X$ has Albanese dimension greater than one. 
\en
\end{lemma}

Before proving this lemma, let us mention that it is possible to verify directly that groups that are virtually a nonabelian surface groups cannot satisfy any of the three cases above.

\begin{proof} We will first show $(1) \Leftrightarrow (2)$, and then $\lnot (2) \Leftrightarrow \lnot (3)$.

$(1) \Rightarrow (2)$:  If $(1)$ holds we can write the short exact sequence \begin{equation} \label{eq:ext}1 \rightarrow \ker ~ \phi\rightarrow G \stackrel{\phi}{\rightarrow} \Z \rightarrow 1 \end{equation} with $\ker ~ \phi$ finitely generated. Hence $G$ is both an ascending and descending HNN extension.  It follows that the rational rays determined by both $\pm \phi \in H^1(G)$ are contained in $\Sigma^1(G)$.  

$(2) \Rightarrow (1)$:  Let us assume that $\Sigma^1(G)$ is nonempty. As  $\Sigma^1(G)$ is
 open, we can assume that there exists a primitive class $\phi \in H^{1}(G)$ whose projective class is determined by a rational ray in $\Sigma^1(G)$. 
K\"ahler groups cannot  be written as \textit{properly} ascending or descending extensions, i.e.\ ascending extensions are also descending and viceversa. (This was first proven in \cite{NR08}; see also \cite{FV16}.) But this is  to say that $G$ has the form of Equation
 (\ref{eq:ext}) with $\ker ~ \phi$ finitely generated.

To show the equivalence of $\lnot(2)$ and $\lnot(3)$, we start by recalling Delzant's description of the BNS invariant of a K\"ahler group $G$. Let $X$ be a K\"ahler manifold with $G = \pi_1(X)$. The collection of irrational pencils $f_{\a}\colon X \to \Sigma_{\a}$ such that the orbifold fundamental group $\Gamma_{\a} := \pi_1^{\orb}(\Sigma)$ is a cocompact Fuchsian group, is finite up to holomorphic automorphisms of the base (see \cite[Theorem 2]{De08}).  (In the language of orbifolds, these are the  holomorphic orbifold maps with connected fibers from $X$ to hyperbolic Riemann orbisurfaces.) 
The pencil maps give, in homotopy, a finite collection of epimorphisms with finitely generated kernel $f_{\a}\colon G \to \Gamma_{a}$. Then \cite[Th\'eor\`eme 1.1]{De10} asserts that the complement of $\S^1(G)$ in $S(G)$  (i.e.\ the set of so--called exceptional characters) is given by
\begin{equation} \label{eq:ex}  S(G) \setminus \S^1(G) =  \mbox{$\bigcup_{\a}$}[f^{*}_{\a} H^{1}(\Gamma_{\a};\R) - \{0\}],\end{equation} where we use the brackets $[ \cdot ]$ to denote the image of a subset of $(H^{1}(G;\R) \setminus \{0\})$ in $(H^{1}(G;\R) \setminus \{0\})/\R_{>0}$. (Note, instead, that genus $1$ pencils without multiple fibers do not induce exceptional characters.)

$\lnot (3) \Rightarrow \lnot (2)$: The negation of $(3)$ asserts that $X$ has either Albanese dimension zero, in which case $S(G)$ is empty, or it has an Albanese pencil $f\colon X \to \Sigma$ with $\G := \pi_{1}^{\orb}(\S)$ cocompact Fuchsian, in which case  $f^{*}H^{1}(\Gamma;\R) = H^1(G;\R)$. In either case $\S^1(G)$ is empty, i.e.\ $\lnot (2)$ holds.

$\lnot (2) \Rightarrow \lnot (3)$: If the set $\S^1(G) \subseteq S(G)$ is empty, either $S(G)$ is empty (i.e.\ $b_1(G) = 0$) whence $G$ has Albanese dimension zero, or by Equation (\ref{eq:ex}) there exists an irrational pencil $f\colon X \to \Sigma$, with $\G := \pi_{1}^{\orb}(\S)$ cocompact Fuchsian, inducing an isomorphism $f^{*}\colon H^{1}(\Gamma;\R) \to H^{1}(G;\R)$. (The union of finitely many \textit{proper} vector subspaces of $H^1(G;\R)$  cannot equal $H^1(G;\R)$.) 
Such a pencil is then the Albanese pencil of $X$.  \end{proof}

In our understanding, irregular K\"ahler manifolds whose Albanese dimension is smaller than their dimension are ``nongeneric", and their study should reduce, through a sort of dimensional reduction induced by the Albanese map, to the study of lower dimensional spaces (see e.g.\ \cite{Cat91}). In that sense, we think of groups that, together with their finite index subgroups, fail to satisfy the equivalent conditions (1) to (3) of Lemma \ref{lemma:equivalent} as nongeneric. 

Examples of K\"ahler groups with $a(G) > 1$ obviously abound.  It is less obvious to provide examples of K\"ahler groups  that have a jump in Albanese dimension,  i.e.\ $a(G) = 1$ but $va(G) > 1$. Before doing so,  we can make an observation about the geometric meaning of such occurrence.  Given an irrational pencil $f\colon X \to \Sigma$, its \textit{relative irregularity} is defined as \[ q_{f} = q(X) - q(\Sigma)\] (where the irregularity of $\S$ equals its genus). The Albanese pencil occurs exactly when $q_{f} = 0$. The notion of relative irregularity allows us to tie the notion of virtual Albanese dimension larger than one with the more familiar notion of virtual positive Betti number (or more properly, in K\"ahler context, virtual irregularity): a K\"ahler group with $a(G) = 1$ has $va(G) > 1$ if and only if there is a lift ${\tilde f}\colon {\widetilde X} \to {\widetilde \Sigma}$ of the Albanese pencil which is \textit{irregularly fibered}, i.e.\ $q_{\tilde f} > 0$.

A fairly simple class of examples comes from groups of type  $G = \pi_{g} \times K$ where $\pi_{g}$ is the fundamental group of a genus $g > 1$ surface and $K$ the fundamental group of a hyperbolic
 orbisurface of genus $0$, so that $b_1(K) = 0$. The group $K$ is K\"ahler 
(e.g.\ it is the fundamental group of an elliptic surface with enough multiple fibers and multiplicity), hence so is  $G$. As $H^1(G;\Z) = H^{1}(\pi_{g};\Z)$, the Albanese dimension of $G$ must be one. On the other hand, $K$ has a finite index subgroup that is the fundamental group of a genus $h > 1$ surface, hence $G$ is virtually $\pi_{g} \times \pi_{h}$. The algebraic surface $\S_g \times \S_g$ has Albanese dimension $2$, so by Proposition \ref{prop:samealb} the virtual Albanese dimension of $G$ is greater than one. (Group theoretically, this can be seen as consequence of \cite[Theorem 7.4]{BNS87}, which asserts that for cartesian products of groups with positive $b_1$ the BNS invariant is nonempty.) 

Less trivial examples with $a(G) = 1$ but $va(G) > 1$ come from bielliptic surfaces. These possess an Albanese pencil of genus $1$ without multiple fibers, but their fundamental groups are virtually $\Z^4$, hence have virtual Albanese dimension $2$ (see \cite[Section V.5]{BHPV04}). More sophisticated examples of K\"ahler surfaces (hence groups) with $a(G) = 1$ that are   finitely covered by the product of curves of genera bigger than one are discussed in \cite[Theorem F]{Cat00}. 

\section{Groups with Virtual Albanese Dimension One} \label{sec:va1}

In this section, we will discuss groups with $va(G) = 1$. Familiar examples of K\"ahler groups with $va(G) = 1$ are given by surface groups, and other simple examples arise as follows. We say that a group $G$ is \textit{commensurable} to a surface group if it admits a (normal) finite index subgroup $H \leq_{f} G$ that  admits an epimorphism $f \colon H \to \G$ to a surface group $\G$ with finite kernel, namely for which there exists an exact sequence \[ 1 \rightarrow F \rightarrow H \stackrel{f}{\rightarrow} \G \rightarrow  1\] with $\G$ a surface group and $F$ finite. 
The map $f \colon H_1(H) \to H_1(\G)$ is then an epimorphism with torsion kernel, i.e. if $G$ (hence $H$) is a K\"ahler group, \  $f \colon H \to \G$ represents in homotopy the Albanese map. 
Quite obviously, each finite index subgroup of $G$ will also be commensurable to a surface group, hence the virtual Albanese dimension of $G$  equals one. Note that when $G$ is commensurable to a surface group, it actually admits a (normal) finite index subgroup which is a surface group, namelly they are virtually surface groups: see \cite[Proposition 3.1]{Jo98} for a proof.

We want to analyze the picture so far in comparison with the situation for $3$--manifold groups. The class of irreducible $3$--manifolds that are not virtually fibered is limited (it is composed entirely by graph manifolds). One may contemplate that, similarly, the K\"ahler counterpart to that class contains only the obvious candidates, namely manifolds whose fundamental group is virtually a surface group. Proposition \ref{prop:product} below guarantees that this is not quite the case. The starting point is the existence of infinite K\"ahler groups (such as $Sp(2n,\Z)$, $n > 1$) with $vb_1(G)$, hence $va(G)$, equal to zero. These do not have a counterpart in dimension $3$.  

\begin{proposition} \label{prop:product} Let $\Gamma$ be the fundamental group of a genus $g > 0$ surface and let $K$ be an infinite K\"ahler groups with $va(K) = 0$; then $va(K \times \G) = a(K \times \Gamma) = 1$ and $K \times \G$ is not virtually a surface group. \end{proposition}

\begin{proof} The projection map $f\colon K \times \G \to \G$ is the Albanese map, hence $a(K \times \Gamma) = 1$.  We claim that as $vb_{1}(K) = 0$, the virtual Albanese dimension of $K \times \G$ is one. In fact, for any normal subgroup $H \leq_{f} K \times \G$ we have from (\ref{eq:diag}) \[ \xymatrix@=9pt{ 
 1\ar[rr] & & \Delta \ar[rr]\ar[dd]\ar[dr] & & H \ar'[d][dd]\ar[dr]  \ar[rr]^{\widetilde f} & & \Lambda \ar[rr]\ar[dr]\ar'[d][dd] & & 1
\\ & & & H_1(\Delta)  \ar[rr] & & H_1(H) \ar[rr] &  & H_{1}(\Lambda )
\\ 1 \ar[rr] &  & K \ar[rr] & & K \times \G  \ar[rr]^{f} & & \G \ar[rr] & & 1.
  } \] As $\Delta \leq_{f} K$ and $vb_{1}(K) = 0$, $H_1(\Delta)$ is torsion, hence  $\mbox{Ker}(H_1(H) \to H_1(\Lambda)) = \mbox{Im}(H_1(\Delta) \to H_{1}(H))$ is torsion as well. It follows that the map ${\widetilde f} \colon H \to \Lambda$ is, in homotopy, the Albanese map, hence $a(H) = 1$. 

There are many ways to show that $K \times \G$ does not have a finite--index subgroup which is a surface group; for instance a quick proof can be obtained using standard properties of $L^2$--Betti numbers.
\end{proof}

 The examples of Proposition \ref{prop:product}  guarante that the class of K\"ahler groups that do not virtually admit an epimorphism to $\Z$ with finitely generated kernel is more variegated  than its counterpart in the $3$--manifold world.  We should, however, qualify
 this result.  These examples build on the existence of infinite K\"ahler groups with $vb_1 = 0$, and leverage on the fact that we can take products of finitely presented groups, as the class of K\"ahler manifolds is closed under cartesian product. Neither of these phenomena has a counterpart in the realm of $3$--manifolds. It is perhaps not too greedy to ask for examples of K\"ahler groups with $va(G) = 1$ in a realm where simple constructions as the one of Proposition \ref{prop:product} are tuned out. 

As we are about to see, this occurs in the case of aspherical surfaces. That's a level playing field with irreducible $3$--manifolds with $b_1 > 0$, which are as well aspherical.

\begin{question} \label{qu:bq} Does there exist a group $G$  with $b_1(G) > 0$ that is the fundamental group of  an aspherical K\"ahler surface and does not virtually algebraically fiber?
 \end{question}

The reason why such an example would be appealing comes from the fact that, perhaps going to a finite index subgroup, the fundamental group $G$ of an aspherical K\"ahler surface  with $va(G) = 1$ is a $4$--dimensional Poincar\'e duality group whose Albanese pencil $f\colon X \to \S$ determines a short exact sequence of finitely generated groups
\begin{equation} 1 \rightarrow K \rightarrow G \rightarrow \G \rightarrow  1, \end{equation} with $\G$ a surface group and $K$ finitely generated. As remarked by Kapovich in \cite{Ka98}, it is a theorem of Hillman that either $K$ is itself a (nontrivial) surface group, or it is {\em not} finitely presented (see e.g.\ \cite[Theorem 1.19]{Hil02}). In either case, a construction like the one in Proposition \ref{prop:product} (or a twist thereof) is excluded. Constructions of this type may not be easy to find. For instance, Stover gives in  (\cite[Theorem 2]{Sto15}) an example of a K\"ahler group, a cocompact arithmetic lattice in $PU(2,1)$, which algebraically fibers. The same is true for the Cartwright--Steger surface (\cite{CS10}), another ball quotient with  $b_1(G) = 2$  and whose Albanese map (as shown in \cite[Corollary 5.3]{CKY15}) has no multiple fiber: by the discussion in the proof of Lemma \ref{lemma:equivalent}  $G$ itself has BNS invariant $\S^1(G)$ equal to the entire $S(G)$ and all epimorphisms to $\Z$ have finitely generated kernel. (According to the introduction to \cite{Sto15}, this  fact was known also to Stover and collaborators.) 

We want now  to show how Question \ref{qu:bq} ties with the study of coherence of fundamental groups of aspherical K\"ahler surfaces, which was initiated in \cite{Ka98,Ka13} using the aforementioned result of Hillman, and further pursued in \cite{Py16}. The outcome of these articles is that, with the obvious exceptions, most aspherical K\"ahler surfaces can be shown to have non--coherent fundamental group. (See \cite[Theorem 4]{Py16} for a detailed statement, which uses a notation slightly different from ours.) Drawing in part from the same circle of ideas of these references (as well as a minor extension of the work in \cite{Ko99}) we can prove the following result, which improves on the existing results insofar as it further narrows possible coherent fundamental groups to finite index subgroups of the fundamental group of Kodaira fibrations with virtual Albanese dimension one. (A pencil on a K\"ahler surface is called a Kodaira fibration if it is smooth and not isotrivial. Its fibers and base have genera at least $3$ and $2$ respectively, see \cite[Section V.14]{BHPV04}.) The new tool yielding the additional noncoherence results is the existence of algebraic fibrations on $G$.

\begin{theorem} \label{thm:koda} Let $G$ be a group with $b_1(G) > 0$  which is the fundamental group of  an aspherical K\"ahler surface $X$; then $G$ is not coherent, except for the case where it is virtually the product of $\Z^2$ by a surface group, or perhaps for the case where $X$ is finitely covered by a Kodaira fibration of virtual Albanese dimension one. \end{theorem}
\begin{proof} 
If $G$ has $va(G) > 1$, let $H \leq_{f} G$ be a subgroup, corresponding to a finite $n$--cover ${\widetilde X}$ of $X$,  which algebraically fibers.  Let \begin{equation} 1 \rightarrow \ker ~ \phi\rightarrow H \stackrel{\phi}{\rightarrow} \Z \rightarrow 1 \end{equation} with $\ker ~ \phi$ finitely generated represent an algebraic fibration.
By \cite[Theorem 4.5(4)]{Hil02} the finitely generated group $\ker ~ \phi$ has type $FP_2$ if and only if the Euler characteristic $e({\widetilde X}) = n e(X) = 0$. A finitely presented group has type $FP_2$. It follows that $\ker ~ \phi \leq G$ is finitely generated but not finitely presented, hence $G$ is not coherent, unless $e(X) = 0$.  If $e(X) = 0$ the classification of compact complex surfaces (see e.g.\ \cite[Table 10]{BHPV04}) entails that $X$ admits an  irrational pencil with elliptic fibers. As $e(X) = 0$ the Zeuthen--Segre formula (see e.g.\ \cite[Proposition III.11.4]{BHPV04}) implies that the only singular fibers can be multiple covers of an elliptic fiber, hence $X$ is finitely covered by a torus bundle.  An holomorphic fibration with smooth fibers of genus $1$ is also isotrivial (i.e. all  fibers are isomorphic), namely a holomorphic fiber bundle, see [BHPV04, Section V.14]. We can invoke then [BHPV04, Sections V.5 and V.6] to deduce that some finite cover of X is a product $T^2 \times \Sigma_g$. 
In this case, the fundamental group is virtually $\Z^2 \times \G$, with $\G$ a surface group. By Theorem B of \cite{BHMS02} any finitely generated subgroup $L \leq \Z^2 \times \G$ has a finite index subgroup that is the product of finitely generated subgroups of each factor. As surface groups are coherent, $L$ must be finitely presented, hence in this case the fundamental group of $X$ is coherent.

Non--obvious examples of coherent groups satisfy therefore $va(G) = 1$, but we can further narrow down this occurence: If $G$ has $va(G) = 1$, perhaps going to a finite index subgroup, it is a $4$--dimensional Poincar\'e duality group whose Albanese pencil  $f\colon X \to \S$ determines as usual the short exact sequence of finitely generated groups $1 \rightarrow K \rightarrow G \rightarrow \G \rightarrow  1$, with $\G$ a surface group and $K$ finitely generated.  To deal with this case, we relay on the strategy of \cite{Ka98,Ka13}: by the aforementioned theorem of Hillman \cite[Theorem 1.19]{Hil02} either $K$ is not finitely presented (and $G$ is not coherent) or it is itself a nontrivial surface group.  
In the latter case, we can follow the path of \cite{Ko99} (see also \cite{Hil00}) to complete the proof of the statement.
We first observe that the surface $X$, being aspherical, is homotopy equivalent to a smooth $4$--manifold $M^{4}$, a surface bundle over a surface $F \hookrightarrow M \to \S$, where $\pi_{1}(F) = K$, $\pi_{1}(\S) = \G$. Note that $M$ is uniquely determined by the short exact sequence of its fundamental group, see \cite[Theorem 5.2]{Hil02}.  Next, as $X$ and $M$ are homotopy equivalent, they have the same Euler characteristic $e(X) = e(M) = e(F) \cdot e(\S)$. At this point, the Zeuthen--Segre formula entails that the only nonsmooth fibers of the Albanese pencil could be multiple covers of an elliptic fiber (in particular, $g(F) = 1$). The assumption that $K$ is finitely generated excludes the presence of multiple fibers (\cite[Lemma 4.2]{Cat03}). This means that Albanese pencil $f\colon X \to \S$ is smooth (i.e. a  holomorphic fibration of maximal rank), namely $X$ is actually a surface bundle over a surface, in particular it is diffeomorphic to $M$. If the pencil was isotrivial (i.e.\ all fibers isomorphic), it would be a holomorphic fiber bundle, and we would conclude as above that  some finite cover of $X$ is a product, whence $va(X) = 2$ and $va(G) > 1$. The statement follows.
\end{proof}

In summary, the existence of non-obvious examples of coherent fundamental groups of aspherical 
K\"ahler surfaces hinges on an affirmative answer to Question \ref{qu:bq}.

\begin{remarks} 
\begin{enumerate}
\item Note that the proof of the first case of Theorem \ref{thm:koda} applies \textit{verbatim} also in the case of aspherical surfaces with $a(G) = 1$ that algebraically fiber; in particular, this entails that the Cartwright--Steger surface and the surfaces described in 
\cite[Theorem 1.2]{DiCS17} have non--coherent fundamental groups. Those groups are 
torsion--free lattices $G \leq PU(2,1)$, and the surfaces are ball quotient $B^{2}_{\C}/G$, i.e. complex hyperbolic surfaces. This was implicitly known (for slightly different reasons) also from \cite{Ka98}; we point out that for the argument above we don't need to invoke \cite{Liu96}.
\item It is not difficult to prove the existence of Kodaira fibrations of Albanese dimension one (which, by the Hochschild--Serre spectral sequence, are surface bundles whose coinvariant homology of the fiber $H_1(K;\Z)_{\G}$ has rank zero): in fact, the ``generic" Kodaira fibration arising from a holomorphic curve in a moduli space of curves has Albanese dimension one. However, this seems to have no obvious consequences for the discussion above: we are not aware of the existence of Kodaira fibrations of virtual Albanese dimension one. We can add that, even if such surfaces did exist, we cannot decide if their fundamental groups are coherent. For instance, it is not obvious whether they may contain $F_2 \times F_2$ as subgroup. 
\item  The first examples of Kodaira fibrations, due to Kodaira and Atiyah, actually carry two  inequivalent structures of Kodaira fibrations, hence are guaranteed to have Albanese dimension two. By the above, their fundamental group is not coherent. The same result applies for the doubly fibered Kodaira fibrations constructed in \cite{CR09,LLR17}. Note that all doubly fibered surfaces bundles with positive signature are of type III in Johnson's trichotomy (i.e. the monodromy representation $\rho : \pi_1(\S) \to Mod(F)$ is injective, and this holds for \textit{all} fibrations $F \hookrightarrow X \to \Sigma$), see \cite{Jo94}. It follows that the K\"ahler group $G = \pi_1(X)$ injects as subgroup of the corresponding mapping class group for the once--punctured fiber  $Mod(F^1)$. The smallest fiber genus attained equals $g(F) = 4$, which is optimal, see \cite[Table 3]{LLR17}.
\end{enumerate} \end{remarks}

It is interesting to flesh out one consequence of the proof of Theorem \ref{thm:koda} that gives some information on higher BNS--type invariants of some K\"ahler groups. Precisely, we will consider the homotopical BNSR invariant $\S^2(G) \subseteq \S^1(G) \subseteq S(G)$ introduced in \cite{BR88}. We will not need the definition of these invariant and we will limit ourselves to mention the well--known fact (see e.g. \cite[Section 1.3]{BGK10}) that, using the notation preceding Lemma \ref{lemma:equivalent},  given a primitive class $\phi \in H^1(G)$, the kernel $\ker ~ \phi \leq G$ is of type $F_2$ (namely, finitely presented) if and only if the rational rays determined by both $\pm \phi \in H^1(G)$ are contained in $\Sigma^2(G)$. 

While we don't know a way to get complete information on the full invariant $\Sigma^2(G)$, the ingredients of the proof of Theorem \ref{thm:koda} is sufficient to entail the following lemma, that \textit{per se} refines the previous result of non--coherence, and is possibly one of the first results on higher invariants of 
K\"ahler groups, besides the case of direct products:

\begin{proposition} \label{proposition:bnsr} The fundamental group $G$ of an aspherical K\"ahler surface $X$ of strictly positive  Euler characteristic with Albanese dimension two has BNSR invariants satisfying the inclusions \[  \S^2(G) \subsetneq \S^1(G) \subseteq S(G). \] \end{proposition}

\begin{proof} The point of this statement is that the first inclusion is strict. For sake of clarity, we review the argument we used in the proof of Theorem \ref{thm:koda}: the condition on the Albanese dimension implies that $G$ algebraically fibers, for some primitive class $\phi \in H^1(G)$. As the Euler characteristic of $X$ is strictly positive, Hillman's Theorem (\cite[Theorem 4.5(4)]{Hil02}) entails that $\ker ~ \phi$ is not $FP_2$, nor \textit{a fortiori} finitely presented. 
\end{proof}

 Note that the same conclusion of the lemma holds, even when $a(G) = 1$, as long as the fundamental group algebraically fibers, e.g. for the Cartwright--Steger surface. Perhaps more importantly, the corollary applies to the aforementioned Kodaira fibrations defined by Kodaira and Atiyah. Topologically, these are surface bundles over a surface, so that their fundamental groups are nontrivial extension of a surface group by a surface group. Higher BNSR invariants of direct products of surfaces groups are (to an extent) well understood by purely group theoretical reasons. Instead, Proposition \ref{proposition:bnsr} seems to be the first result of that type for nontrivial extensions.

\begin{remark} The reader familiar with BNSR invariants may notice that we are just shy of  being able to conclude that the fundamental group of aspherical K\"ahler surfaces of positive  Euler characteristic has empty $\S^2(G)$. We conjecture that this is true. (The conjecture holds true whenever $\S^2(G) = - \S^2(G)$.)  We mention also that the statement of Proposition \ref{proposition:bnsr} remains true if we consider the homological BNSR invariant $\S^2(G;\Z)$ of  \cite{BR88}.
 \end{remark}

We will discuss now a result that further ties groups with virtual Albanese dimension one and surface groups,  asserting that the  Green--Lazarsfeld sets of such groups coincide (up to going to a finite index subgroup) with those of their Albanese image. 

The Green--Lazarsfeld sets of a K\"ahler manifold $X$ (and, by extension, of its fundamental group $G$) are subsets of the \textit{character variety} of $G$,  the complex algebraic group defined as ${\widehat G} := H^{1}(G;\C^{*})$ . The Green--Lazarsfeld sets $W_{i}(G)$ are defined as the collection of cohomology jumping loci of the character variety, namely \[ {W}_{i}(G) = \{ \xi \in {\widehat G} ~ | \rk H^1(G;\C_{\xi}) \geq i \},\] nested by the \textit{depth} $i$: ${W}_{i}(G) \subseteq {W}_{i-1}(G) \subseteq \ldots \subseteq {W}_{0}(G) = {\widehat G}$. 

For K\"ahler groups the structure of ${W}_{1}(G)$ is well--understood. The projective case appeared in \cite{Si93} (that refined previous results of \cite{GL87,GL91}); this result was then extended to the K\"ahler case in \cite{Cam01} (see also \cite{De08}). 
Briefly, $W_1(G)$ is the union of a finite set of isolated torsion characters and the inverse image of the Green--Lazarsfeld set of   hyperbolic orbisurfaces under the finite collection of pencils of $X$ with hyperbolic base. 

If $X$ has Albanese dimension one, the Albanese map $f\colon G \to \G$ 
induces an epimorphism $f_{*}\colon  H_{1}(G) \to H_{1}(\G)$. Therefore, we have an induced  isomorphism of the connected components of the character varieties containing the trivial character
\begin{equation} \label{eq:char}  {\widehat f}\colon {\widehat \G}_{\hat 1} \stackrel{\cong}{\longrightarrow} {\widehat G}_{\hat 1} \end{equation}
where for a group $G$, we denote the connected component of the character variety containing the trivial character ${\hat 1} \colon G \to \C^*$ as $ {\widehat G}_{\hat 1}$.

The next theorem, a restatement of Theorem \ref{thm:gli}, shows that if $X$ has virtual Albanese dimension one, then after perhaps going to a cover, the map ${\widehat f}$ restricts to an isomorphism of the Green--Lazarsfeld sets.

\begin{theorem} \label{thm:gls} Let $X$ be a K\"ahler manifold with $va(X) = 1$.
 Up to going to a finite index normal subgroup if necessary,  the Albanese map $f\colon G \to \G$ induces an isomorphism
  \[  {\widehat f}: {W}_{i}({\G} ) \stackrel{\cong}{\longrightarrow}  {W}_{i}(G) \] of the Green--Lazarsfeld sets. 
  \end{theorem} 

\begin{proof} Up to going to a finite cover, we can assume that $X$ admits an Albanese pencil $f \colon X \to \S$. 
 Moreover, as every cocompact Fuchsian group of positive genus admits a finite index normal subgroup which is a honest surface group, we can also assume that, after going to a further finite cover if necessary, the Albanese pencil doesn't contain any multiple fibers. In particular, $H_1(\S)$ will be torsion--free. Without loss of generality, by going to the normal core of the associated finite index subgroup, we can always assume that the cover  is regular. Summing up, after possibly going to a finite cover the Albanese map, in homotopy, is an epimorphism $f \colon G \to \G$ where $\G = \pi_1(\S)$ is a genus $g(\G)$ surface group.

 The Green--Lazarsfeld sets $W_{i}(\Gamma)$ for a surface group are determined in \cite{Hir97}, and are  given by
\begin{equation} \label{eq:glsets} W_{i}(\G)\,\, =\,\, \left\{ \begin{array}{lll} {\widehat \G} 
    & \mbox{if $1 \leq i\leq 2g(\G) - 2$,} \\  {\hat 1} & \mbox{if $ 2g(\G) - 1 \leq i \leq  2g(\G)$},  \\  \emptyset & \mbox{if $i \geq  2g(\Gamma) + 1$}. \end{array} \right. 
\end{equation}

Given $\rho \in {\widehat \Gamma}$, surjectivity of $f\colon G \to \G$ implies by general arguments (see  e.g.\ \cite[Proposition 3.1.3]{Hir97}) that $f^{*}\colon H^{1}(\G;\C_{\rho}) \to H^{1}(G;\C_{f^{*}(\rho)})$ is a monomorphism. But we will actually need more, namely that by \cite[Theorem 1.1]{Br02} or \cite[Theorem 1.8]{Br03} $f^{*}$  is
 an isomorphism, except perhaps when $\rho \in {\widehat \G}$ is a torsion character. 

This implies that ${\widehat f}\colon {W}_{i}({\G} ) \to  {W}_{i}(G)$ is an injective map, and it will fail to preserve the depth (i.e.\ dimension of the twisted homology) only for torsion characters.  

 Consider the short exact sequence of groups
  \[ 1 \longrightarrow {\widehat G}_{\hat 1} \longrightarrow {\widehat G} \stackrel{t}{\longrightarrow} \Hom(\tor H_1(G);\C^{*}) \longrightarrow 1, \] where  ${\widehat G}_{\hat 1}$ refers as above to the component of ${\widehat G}$ connected to the trivial character ${\hat 1}$. 

By \cite[Theorem 0.1]{GL91} all irreducible positive dimensional components of  ${W}_{i}(G)$ are  inverse images of the Green--Lazarsfeld set of the  hyperbolic orbisurfaces. By Proposition \ref{prop:unique}, the Albanese pencil is unique, hence ${\widehat f}({W}_{i}({\G} )) =  {\widehat G}_{\hat 1} \subseteq W_{i}(G)$ (for $1 \leq i \leq 2q(G)-2$) is the only positive dimensional component, that will occur if (and only if ) $q(G) = g(\G) \geq 2$.

We now have the following claim.

\begin{claim}
For all $i \geq 1$, $W_i(G) \setminus {\widehat f}({W}_{i}({\G} ))$ is composed of torsion characters. 
\end{claim}

If $q(G) = 1$, this follows immediately from  \cite[Th\'eor\`eme 1.3]{Cam01}, as in this case $W_1(\G)$  is torsion. If $q(G) \geq 2$, we need a  bit more work: again by \cite[Th\'eor\`eme 1.3]{Cam01} and the above, we have that
\[ W_1(G) =  {\widehat G}_{\hat 1}  \mbox{ $\bigcup$ }  Z \] where $Z$ is a finite collection of torsion characters, that we will assume to be disjoint from ${\widehat G}_{\hat 1}$. By definition $W_{i}(G)\subseteq W_1(G)$. This implies,  by the aforementioned result of Brudnyi, that:  

-- if $1 \leq i \leq  2q(G)-2$ the isolated points of $W_i(G)$  are contained in $Z$; 

-- if $i \geq 2q(G) -1$, they are either contained in $Z$,  or are torsion characters in  ${\widehat G}_{\hat 1}$.
In either case,  $W_i(G) \setminus {\widehat f}({W}_{i}({\G} ))$ is composed of torsion characters as claimed. This concludes the proof of the claim.

All this, so far, is a consequence of the fact that $X$ has Albanese dimension one. Now we will make use of the assumption on the virtual Albanese dimension to show that $W_i(G) \setminus {\widehat f}({W}_{i}({\G} ))$ is actually empty.

In order to prove this, recall the formula for the first Betti number for finite regular abelian covers of $X$, as determined in \cite{Hir97}: Given an epimorphism $\a\colon G \to S$ to a finite abelian group, and following the notation from the diagram in (\ref{eq:diag}), the finite cover $H$ of $X$ determined by $\a$ has first Betti number  \begin{equation} \label{eq:b1h}  b_1(H) = \sum_{i \geq 1} |W_{i}(G) \cap {\widehat \a}({\widehat S})|.\end{equation} This formula says that a character $\xi\colon G \to \C^{*}$ such that $\xi \in W_i(G)$ contributes with multiplicity equal to its depth to the Betti number of the cover defined by $\a\colon G \to S$ whenever it factorizes via $\a$. Similarly, the corresponding cover $\Lambda$ of $\Sigma$ has first Betti number  
  \begin{equation} \label{eq:b1l} b_1(\Lambda) = \sum_{i \geq 1} |W_{i}(\G) \cap {\widehat \b}({\widehat{S/\a(K)}})|.\end{equation}
   Consider a character $\rho \in W_i(\G) \cap {\widehat \b}({\widehat{S/\a(K)}})$.  We have a commutative diagram
\[ \xymatrix@C1.2cm@R0.65cm{ G\ar[d]_\a \ar[r]^f & \G \ar[d]_\b \ar[ddr]^\rho \\
S \ar[r] & S/\a(K) \ar[dr]  \\
& & \C^{*}  }\]
whence $f^*(\rho)$ factorizes via $\a$.  As ${\widehat f}(W_i(\G)) \subseteq W_{i}(G)$ this implies that $f^{*}(\rho) \in  W_{i}(G) \cap {\widehat \a}({\widehat S})$. We deduce that for any $\a: G \to S$ any contribution to the right hand side of Equation   (\ref{eq:b1l}) is matched by an equal contribution to the right-hand-side of Equation   (\ref{eq:b1h}).

Now assume that, for some $i \geq 1$, there is a {\em nontrivial} character $\xi\colon G \to \C^{*}$ such that $\xi \in W_i(G) \setminus {\widehat f}({W}_{i}({\G} ))$. (The case of the {\em trivial} character ${\hat 1}\colon G \to \C^*$ is
 dealt with in the same way; we omit the details to avoid repetition.)  As shown in the claim  above, such character must be torsion. Because of that, its image is a finite abelian subgroup $S \leq \C^{*}$, i.e.\ $\xi$  factors through an epimorphism $\a_{\xi} :G \to S$. 
This character either lies in $Z$ or it is  of the form $\xi = f^*({\rho})$ for $\rho \in W_{2g(\G) -2}(\G)$ (in which case $i >2g(\G) - 2$). These two cases are treated in slightly different ways:

-- If the case $\xi \in Z$ holds, $\xi$ will give a positive contribution to $b_1(H_{\xi})$ that is not matched by any term in the right-hand-side of Equation   (\ref{eq:b1l}).

-- If the case $\xi = f^{*}(\rho)$ holds, it is immediate to verify that $\rho$ factors through $\beta_{\xi}\colon \Gamma \to S/\a_\xi(K) \cong S$. However, the contribution of $\rho$ to $b_1(\Lambda_{\xi})$ is at least $i - 2g(\G) + 2 > 0$ short of the contribution of $\xi = f^{*}( \rho)$ to $b_1(H_{\xi})$, as \[ \mbox{dim } H^{1}(G;\C_{\xi}) \geq i > \mbox{dim } H^{1}(\G;\C_{\rho}) = 2g(\G) - 2. \]

 In either case, the outcome is that $b_1(H_{\xi}) > b_1(\Lambda_{\xi})$, which violates the assumption that $X$ has virtual Albanese dimension one. 

Summing up, ${\widehat f}\colon W_{i}(\G) \to W_{i}(G)$ is a bijection. These sets coincide therefore with the connected components of the respective character variety (for $i \leq i \leq 2g(\G) - 2$), the trivial character (for $2g(\G) - 1 \leq i \leq 2g(\G)$), and are empty otherwise.  Whenever nonempty, both sets inherit a group structure   
as subsets of the respective character varieties. The map ${\widehat f}$ is the restriction of an homomorphism between these character varieties, hence an isomorphism as stated. 
\end{proof} 

\begin{remarks} 
\begin{enumerate} 
\item The bielliptic surfaces mentioned in Section \ref{sec:ref} are a clean example of the fact that we need more than Albanese dimension one to get the isomorphism of Theorem \ref{thm:gls}. These surfaces admit  genus one Albanese pencils without multiple fibers, and there exists an epimorphism $\alpha\colon G \to S$ with $S$ abelian  and $H= \ker ~ \a \cong \Z^4$ (see \cite[Section V.5]{BHPV04}). This entails that, for some $i \geq 1$, $W_{i}(G)$ is strictly larger than ${\widehat f}(W_i(\Z^2))$ (the difference being torsion characters, contributing to the first Betti number of $H$). 
\item The converse of Theorem \ref{thm:gls} is false: for instance, if $G$ denotes the fundamental group of the Cartwright--Steger surface, $W_{i}(G) = {\widehat f}(W_i(\Z^2))$, but $va(G) > 1$ (see \cite{Sto18}).
\end{enumerate}
\end{remarks}

We want to discuss now a result that may prevent the existence of new simple examples of K\"ahler groups with $va(G) = 1$, under the assumption that the group satisfies residual properties akin to those holding for most irreducible $3$--manifold groups,
namely being virtually RFRS (in particular, that holds for the fundamental group of irreducible manifolds that have hyperbolic pieces in their geometric decomposition). This class of groups was first introduced by Agol in the study of virtual fibrations of $3$--manifold groups: A group $G$ is RFRS if there exists a filtration $\{G_{i}|i \geq 0\}$ of finite index normal subgroups 
$G_{i} \unlhd_{f} G_{0} = G$ with $\bigcap_{i} G_{i} = \{1\}$ whose successive quotient maps $\a_{i}\colon G_{i} \to G_i/G_{i+1}$  factorize through the maximal free abelian quotient: 
 \begin{equation} \label{eq:rfrs} \xymatrix@=4pt{ 
 1\ar[rr] & & G_{i+1} \ar[rr] & & G_{i} \ar[dr]  \ar[rr]^{\a_{i}} & & G_{i}/G_{i+1} \ar[rr] & & 1
\\  & & & & & H_1(G_{i})/\mbox{Tor} \ar[ur] &  &} \end{equation}

Subgroups of the direct product of surface groups and abelian groups are virtually RFRS. The largest source of virtually RFRS group we are aware of is given by subgroups of right--angled Artin groups (RAAGs), that are virtually RFRS by  \cite{Ag08}. However, this class does not give us new examples, as Py proved in \cite[Theorem A]{Py13} that all K\"ahler groups that are subgroups of RAAGs are in fact virtually subgroups of the product of surface groups and abelian groups.

 We have the following, that combined with Lemma \ref{lemma:equivalent} gives Theorem \ref{thm:vrfrs}:
\begin{theorem} Let $G$ be a virtually RFRS K\"ahler group. Then for any K\"ahler manifold $X$ such that $\pi_{1}(X) = G$ either
there exists a finite cover ${\widetilde X}$ of $X$ with Albanese dimension greater than one, or $G$ is virtually a surface group. \end{theorem}

\begin{proof} After going to a suitable finite cover can assume that $X$ has RFRS fundamental group $G$, with associated sequence  $\{G_{i}\}$. From the definition above (see the sequence in (\ref{eq:rfrs})) a nontrivial RFRS group has positive first Betti number. In light of this, it is sufficient to show that if $G$ is a RFRS group with Albanese map $f \colon G \to \G$ and virtual Albanese dimension one, then it is a surface group. (This implies, because of the initial cover to get $G$ RFRS, the theorem as stated.) Recall that  we have a short exact sequence  \[ \xymatrix{ 1 \ar[r] & K \ar[r] & G \ar[r]^{f} & \Gamma \ar[r] & 1} \] where $\Gamma$ can be assumed to be a surface group and $K$ is finitely generated. We claim that if $X$ has virtual Albanese dimension one, then $K$ is actually trivial, i.e.\ $f$ is injective. Let $\gamma \in G$ be a nontrivial element; the assumption that $\bigcap_{i} G_{i} = \{1\}$ implies that there exist an index $j$ such that $\gamma \in G_{j} \setminus G_{j+1}$.
 Consider now the diagram
 \begin{equation} \label{fig:rfrs} \xymatrix@=4pt{ 
 1\ar[rr] & & K \cap G_{j} \ar[rr] & & G_{j} \ar[dr]  \ar[rr]^{f_{j}} \ar[dd]_{\a_j} & & \G_{j}\ar[rr] \ar@{-->}[dl] & & 1
\\  & & & & & H_1(G_{j})/\mbox{Tor} \ar[dl] &  & 
\\  & & & & G_{j}/G_{j+1} \ar[dd] & &  &
\\ \\ & & & & 1  & &  & } \end{equation}
where $f_j\colon G_{j}  \to \Gamma_{j}$  is the restriction epimorphism between finite index subgroups of $G$ and $\Gamma$ respectively determined by $G_{j} \unlhd G$ as in the commutative diagram of  (\ref{eq:diag}). This map represents, in homotopy, the pencil $f_{j}\colon X_{j} \to \Sigma_{j}$ of the cover of $X$ associated to $G_{j}$. By assumption, this pencil is Albanese.
 This entails that there is an isomorphism
 \[ (f_{j})_{*}\colon  H_1(G_{j})/\mbox{Tor} \stackrel{\cong}{\longrightarrow}  H_1(\G_{j}).\] Composing the inverse of this isomorphism with the maximal free abelian quotient of $\G_{j}$ gives the map denoted with a dashed arrow in the diagram of Equation  (\ref{fig:rfrs}). By the commutativity  of that diagram we deduce that the quotient map $\a_{j}\colon G_{j} \to G_{j}/G_{j+1}$ factors through $f_{j}\colon G_{j} \to \Gamma_{j}$. As $\gamma \in G_{j} \setminus G_{j+1}$, the image $\a_{j}(\gamma) \in G_{j}/G_{j+1}$ is nontrivial, hence so is $f_{j}(\gamma)$. This implies that $f(\gamma) =  f_{j}(\gamma) \in \Gamma$ is nontrivial, i.e.\ $f \colon G \to \G$ is injective.
\end{proof}


\end{document}